\documentclass[12pt]{amsart}
\usepackage{graphicx, amssymb}
\title{Integrating curvature: from Umlaufsatz to $J^+$ invariant.}
\author{Sergei Lanzat}
\address{Department of Mathematics, Technion- Israel
Institute of Technology, Haifa 32000, Israel}
\email{serjl@tx.technion.ac.il, polyak@math.technion.ac.il}
\author{Michael Polyak}

\newcommand{\ZZ}{\mathbb{Z}}
\newcommand{\RR}{\mathbb{R}}
\newcommand{\SSS}{\mathbb{S}}
\newcommand{\sss}{\scriptstyle\mathbb{S}}
\newcommand{\minus}{\smallsetminus}
\newcommand{\G}{\Gamma}
\newcommand{\tG}{\widetilde{\Gamma}}
\newcommand{\be}{\begin{itemize}}
\newcommand{\ee}{\end{itemize}}
\def\ind{\operatorname{ind}_\Gamma}
\def\tind{\operatorname{ind}_{\widetilde{\Gamma}}}
\def\rot{\operatorname{rot}}

\newtheorem{thm}{Theorem}
\newtheorem{cor}[thm]{Corollary}
\newtheorem{prop}[thm]{Proposition}
\newtheorem{rem}[thm]{Remark}

\begin{document}

 \begin{abstract}
Hopf's Umlaufsatz relates the total curvature of a closed immersed 
plane curve to its rotation number. While the curvature of a curve 
changes under local deformations, its integral over a closed curve 
is invariant under regular homotopies. 
A natural question is whether one can find some non-trivial densities on 
a curve, such that the corresponding integrals are (possibly after some 
corrections) also invariant under regular homotopies of the curve in 
the class of generic immersions. We construct a family of such densities 
using indices of points relative to the curve.  
This family depends on a formal parameter $q$ and may be considered 
as a quantization of the total curvature. The linear term in the Taylor 
expansion at $q=1$ coincides, up to a normalization, with Arnold's 
$J^+$ invariant. This leads to an integral expression for $J^+$.
 \end{abstract}

\thanks{Both authors were partially supported by the ISF grant 1343/10.}

\maketitle

Let $\G$ be a closed oriented immersed plane curve
$\G:\SSS^1\to\RR^2$. One of the fundamental notions related to $\G$
is its {\em curvature} $\kappa$. Another important notion is that of
a {\em rotation number} (or Whitney winding number) $\rot(\G)$, i.e. 
the number of turns made by the tangent vector as we follow $\G$ 
along its orientation.

Hopf's Umlaufsatz \cite{H} is one of the simplest versions of the
Gauss-Bonnet theorem and one of the fundamental theorems in the
theory of plane curves. It relates two different types of data:
local geometric characteristic of a plane curve -- its curvature
$\kappa$ -- and a global topological characteristic -- its rotation
number $\rot(\G)$. Although the curvature of a plane curve changes
under local deformations, the theorem states that its average
(integral) over a closed curve is invariant under homotopies in the
class of immersed curves:
\begin{thm}[Hopf's Umlaufsatz]\label{thm:umlauf}
\begin{equation}\label{eq:umlauf}
\frac{1}{2\pi}\int_{\sss^1} \kappa(t)\, dt=\rot(\G)
\end{equation}
\end{thm}

A natural question is whether one can find some natural densities
$\rho$ on $\G$ such that the average
$\displaystyle\int_{\sss^1}\kappa(t)\rho(t)\,dt$ is
(possibly after some corrections) also invariant under local
deformations of $\G$. Since the rotation number is (up to
normalization) the only invariant of $\G$ in the class of immersed
curves, we cannot expect such an expression to remain invariant 
under arbitrary homotopies. We can hope, however, that the result 
is invariant under regular homotopies in the class of {\em generic} 
immersions, i.e. immersions with a finite set $X$ of transversal 
double points as the only singularities. Invariants of such a type 
were originally introduced by Arnold \cite{A} and include the 
celebrated $J^\pm$ and $St$ invariants (see \cite{A} for details).

We construct a family of such densities using the {\em index
$\ind(p)$ of a point $p$ relative to $\G$}. Given
$p\in\RR^2\minus\G$, we define $\ind(p)$ as the number of turns
made by the vector pointing from $p$ to $\G(t)$, as we follow $\G$
along its orientation. This defines a locally-constant function on
$\RR^2\minus\G$. See Figure \ref{fig:ind}a. Suppose that $\G$ is
generic. Then we can extend $\ind$ to a $\frac12\ZZ$-valued
function on $\RR^2$. To define $\ind(p)$ for $p\in\G$, average
its values on the regions adjacent to $p$ -- two regions if $p$ is a
regular point of $\G$, and four regions if $p$ is a double point of $\G$.
See Figure \ref{fig:ind}b.
\begin{figure}[htb]
\includegraphics[width=5.0in]{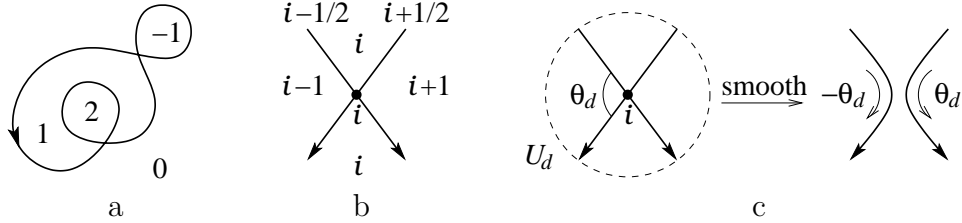}
%\vspace{1in}
\\
a\hspace{1.2in}b\hspace{2.0in}c\hspace{0.5in}
\caption{Indices of points and a smoothing of a double point.
\label{fig:ind}}
\end{figure}
For each double point $d=\G(t_1)=\G(t_2)\in X$,
define $\theta_d\in(0,\pi)$ as the (non-oriented) angle between two
tangent vectors $\G'(t_1)$ and $-\G'(t_2)$. For
$q\in\RR\minus\{0\}$, define
$I_q(\G)\in\RR[q^{\frac12},q^{-\frac12}]$ by
\begin{equation}\label{eq:q-umlauf}
I_q(\G)=\frac{1}{2\pi}\left(\int_{\sss^1} \kappa(t)\cdot
q^{\ind(\G(t))}\, dt -\sum_{d\in X}\theta_d \cdot
q^{\ind(d)}(q^{\frac12}-q^{-\frac12})\right)
\end{equation}

\begin{thm}\label{thm:q-umlauf}
$I_q(\G)$ is invariant under regular homotopies of $\G$ in the class 
of generic immersions.
\end{thm}

\begin{proof}
Note that we can generalize all above notions and formulas to the
case of a multi-component curve $\G:\sqcup_n\SSS^1\to\RR^2$ (by a
summation of indices relative to all components of $\G$).

Let us smooth the original curve $\G$ in each double point
respecting the orientation to get a multi-component curve
$\tG=\cup_n\tG_n$ without double points. Denote by $\tind(p)$ 
the index of a point $p$ relative to $\tG$. Note that values of 
$I_q$ on $\G$ and $\tG$ differ by an easily computable factor 
(which depends only on the regular homotopy class of $\G$ in 
the class of generic immersions). 
Indeed, consider a small neighborhood $U_d$ of a double point 
$d$ of index $i$, see Figure \ref{fig:ind}c. Under smoothing of 
$d$, the total curvature of $\tG\cap U_d$ differs from that of 
$\G\cap U_d$ by $\pm(\pi-\theta_d)$ for the fragment with index 
$i\pm\frac12$, see Figure \ref{fig:ind}c.
Thus the integral part of $I_q$ changes by
$\displaystyle\frac{1}{2\pi}(\pi-\theta_d)(q^{i+\frac12}-q^{i-\frac12})$.
Also, the double point $d$ contributes
$\displaystyle-\frac{1}{2\pi}\theta_d q^i(q^{\frac12}-q^{-\frac12})$
to $I_q(\G)$. Smoothing removes $d$, so this summand disappears from
$I_q(\tG)$. Thus, the total change of $I_q$ under smoothing of $d$
equals $\displaystyle\frac12 q^i(q^{\frac12}-q^{-\frac12})$. Hence
\begin{equation*}
%\label{eq:smoothing of q-umlauf}
I_q(\G)=I_q(\tG)-\frac12\sum_d q^{\ind(d)}(q^{\frac12}-q^{-\frac12})\,.
\end{equation*}
Since $\sum_d q^{\ind(d)}(q^{\frac12}-q^{-\frac12})$ is invariant
under regular homotopies of $\G$ in the class of generic immersions, 
it remains to prove the invariance of $I_q(\tG)=\sum_n I_q(\tG_n)$.

Note that $\tind(\tG(t))$ is constant on each component $\tG_n$
of $\tG$, so
 $$I_q(\tG_n)=\frac{1}{2\pi}\int_{\sss^1}\kappa_n(t)\cdot
 q^{\tind(\tG_n(t))}\,dt=q^{\tind(\tG_n(t))}
 \frac{1}{2\pi}\int_{\sss^1} \kappa_n(t)\,dt$$
and by Umlaufsatz \eqref{eq:umlauf} we get $I_q(\tG_n)=\pm\,
q^{\tind(\tG_n(t))}$, depending on \linebreak $\rot(\tG_n)=\pm1$. 
Thus, $I_q(\tG_n)$ is invariant under regular homotopies of $\tG$. 
But a regular homotopy of $\G$ in the class of generic immersions 
induces a regular homotopy of $\tG$ and the theorem follows.

\end{proof}

Any two immersions with the same rotation number can be connected by
regular homotopy in the class of generic immersions and a finite sequence 
of self-tangency and triple-point modifications, shown in Figure \ref{fig:modif}. 
\begin{figure}[htb]
\includegraphics[width=3.5in]{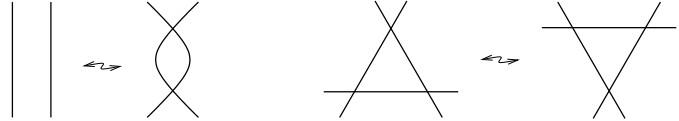}
%\vspace{1in}
\caption{Self-tangency and triple-point modifications.
\label{fig:modif}}
\end{figure}
Depending on orientations and indices of adjacent regions, one can 
distinguish several types of these modifications. Self-tangencies can 
be separated into direct (or dangerous) and opposite (or safe), shown 
in Figure \ref{fig:modif2}a and \ref{fig:modif2}b respectively. 
An index of a self-tangency modification is the index of two new-born 
double points (e.g., modifications in Figure \ref{fig:modif2} are of index $i$). 
Triple-point modifications can be separated into weak (or acyclic) and 
strong (or cyclic), shown in Figure \ref{fig:modif3}a and \ref{fig:modif3}b 
respectively. An index of a triple-point modification\footnote{Our indices 
of modifications differ from the ones of \cite{V} by an $-1$ shift.} is the 
minimum of indices of double points involved in this modification (e.g., 
modifications in Figure \ref{fig:modif3} are of index $i$).
\begin{figure}[htb]
\includegraphics[width=5.0in]{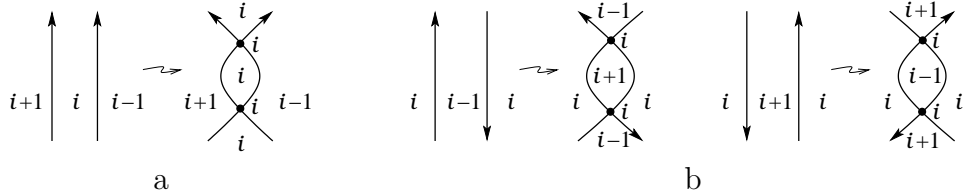}
%\vspace{1in}
\\
a\hspace{2.7in}b\hspace{0.6in}
\caption{Direct and opposite self-tangency modifications of index $i$.
\label{fig:modif2}}
\end{figure}
\begin{figure}[htb]
\includegraphics[width=5.0in]{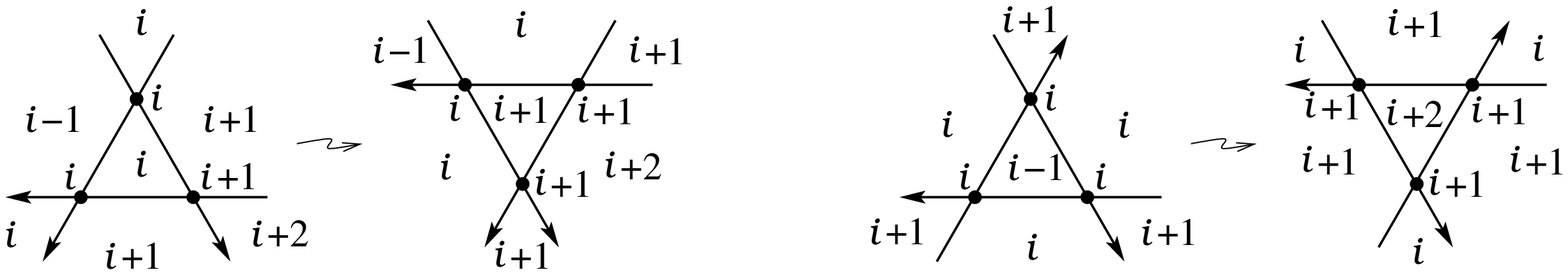}
%\vspace{1in}
\\
a\hspace{2.7in}b
\caption{Weak and strong triple-point modifications of index $i$.
\label{fig:modif3}}
\end{figure}
Invariants of regular homotopy classes of generic immersions are 
unique\-ly determined by their behavior under these modifications, 
together with normalizations on standard curves $K_i$ of 
$\rot(K_i)=i$, $i=0,\pm1,\pm2,\dots$ shown in Figure \ref{fig:standard}.
\begin{figure}[htb]
    \includegraphics[width=5.0in]{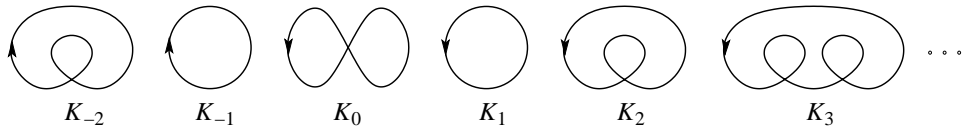}
\caption{Standard curves of indices $0,\pm1,\pm2,\dots$.
\label{fig:standard}}
\end{figure}
Basic invariants $J^\pm$ and $St$ of (regular homotopy classes of)
generic plane curves were introduced axiomatically by Arnold \cite{A}.
In particular, $J^+$ is uniquely determined by the following axioms:
\be
\item $J^+$ does not change under an opposite self-tangency
or triple-point modifications.
\item Under a direct self-tangency modification which increases
the number of double points, $J^+$ jumps by 2.
\item On the standard curves $K_i$ we have $J^+(K_0)=0$ and
$J^+(K_i)=-2(|i|-1)$ for $i=\pm1,\pm2,\dots$.
\ee
In a similar way, $I_q(\G)$ is uniquely determined by the following
\begin{thm}\label{thm:q-properties}
The invariant $I_q(\G)$ satisfies the following properties:
\be
 \item $I_q(\G)$ does not change under opposite self-tangencies.
 \item Under direct self-tangencies of index $i$, the invariant
%shown in Figure \ref{fig:modif2_i}a
$I_q(\G)$ jumps by $-q^i(q^{\frac12}-q^{-\frac12})$.
 \item Under (both weak and strong) triple-point modifications of index $i$, 
%shown in Figure \ref{fig:modif3_i}b, 
$I_q(\G)$ jumps by $-\frac12 q^{i+\frac12}(q^{\frac12}-q^{-\frac12})^2$.
 \item We have $I_q(-\G)=-I_{q^{-1}}(\G)$, where $-\G$ denotes $\G$
 with the opposite orientation.
 \item On the standard curves $K_i$ we have $\displaystyle I_q(K_0)=
 \frac12(q^{\frac12}-q^{-\frac12})$ and 
$\displaystyle I_q(K_i)=
\frac12(i-1)\,q^{\frac32}+\frac12(i+1)\,q^{\frac12}$ for $i=1, 2,\dots$
\ee
\end{thm}
\begin{proof}
A straightforward computation verifies both the behavior of 
$I_q(\G)$ under self-tangencies and triple-point modifications 
and its values on the curves $K_i$. To verify the behavior of 
$I_q(\G)$ under an orientation reversal, note that 
$\operatorname{ind}_{-\G}(p)=-\operatorname{ind}_\G(p)$, 
which corresponds to the involution $q\to q^{-1}$  in terms 
$q^{\ind(\G(t))}$ and $q^{\ind(d)}$ of \eqref{eq:q-umlauf}. 
Also, both terms in \eqref{eq:q-umlauf} change signs: the 
integral due to the change of parametrization, and the sum 
over double points due to the equality 
$q^{\frac12}-q^{-\frac12}=
-\left((q^{-1})^{\frac12}-(q^{-1})^{-\frac12}\right)$.
\end{proof}

Substituting $q=1$ into \eqref{eq:q-umlauf}, we readily obtain
$I_1(\G)=\frac{1}{2\pi}\int_{\sss^1} \kappa(t)\, dt=\rot(\G)$ and
recover the classical Hopf Umlaufsatz, see Theorem \ref{thm:umlauf}.
In this sense, invariant $I_q$ may be considered as a quantization
of the total curvature \eqref{eq:umlauf}. Let us study the next term
$I'_1(\G)$ of the Taylor expansion of $I_q(\G)$ at $q=1$.

\begin{prop}\label{prop:J+}
$I'_1(\G)$ is related to Arnold's $J^+$ invariant by 
$$I'_1(\G)=\frac12(1-J^+(\G))\,.$$
\end{prop}
\begin{proof}
Note that by Theorem \ref{thm:q-umlauf}, $I'_1(\G)$ is invariant 
under regular homotopies of $\G$ in the class of generic immersions. 
Differentiating at $q=1$ expressions for jumps of $I_q(\G)$ in 
Theorem \ref{thm:q-properties} we immediately conclude that 
$I'_1(\G)$ is invariant under opposite tangencies and triple-point 
modifications. Moreover, under direct tangencies, $I'_1(\G)$ jumps 
by $-1$. Thus its behavior under all modifications is the same as 
that of $-\frac12 J^+(\G)$ (up to an additive constant depending 
on $\rot(\G)$). A straightforward computation shows that $I'_1(\G)$ 
takes values $I'_1(K_0)=\frac12$ and $I'_1(K_i)=|i|-\frac12$ for 
$i=\pm1,\pm2,\dots$ on the standard curves $K_i$ and the 
proposition follows.
\end{proof}

Differentiating RHS of \eqref{eq:q-umlauf} at $q=1$ and using
Proposition \ref{prop:J+}, we get

\begin{cor}
The following integral expression for $J^+$ holds:
$$
J^+(\G)=1-\frac{1}{\pi}\left(\int_{\sss^1} \kappa(t)\cdot
\ind(\G(t))\, dt -\sum_{d\in X}\theta_d\right)\ .
$$
\end{cor}

\begin{rem}
An infinite family of invariants, called ``momenta of index'' $M_r$ together 
with their generating function $P_\G(q)\in\ZZ[q,q^{-1}]$ were introduced 
by Viro in \cite[Section 5]{V}. A careful check of their behavior under 
self-tangencies and triple-point modifications, together with their values 
on the standard curves $K_i$, allow one to relate $P_\G(q)$ to $I_q(\G)$ 
as follows:
$$P_\G(q)=(q^{\frac12}-q^{-\frac12})I_q(\G)+1+
\frac12\sum_{d\in X}q^{\ind(d)}(q^{\frac12}-q^{-\frac12})^2\ .$$
\end{rem}
\begin{rem}
Our choice of the function $q^{\ind}$ in the integral part of 
\eqref{eq:q-umlauf} was motivated by considerations of conciseness 
and convenience.
In fact, one can use an arbitrary function $f:\frac12\ZZ\to\RR$ of $\ind$ 
instead of $q^{\ind}$ (with an appropriate change of the correction 
term) to produce an invariant.  
Namely, repeating the proof of Theorem \ref{thm:q-umlauf}, one 
can show that
\begin{multline*}
 Z(f,\G)=\frac{1}{2\pi}\int_{\sss^1} \kappa(t)\cdot
f\left(\ind(\G(t))\right)\, dt - \\
\frac{1}{2\pi}\sum_{d\in X}\theta_d \cdot
\big(f(\ind(d)+1/2)-f(\ind(d)-1/2)\big)
\end{multline*}
is an invariant of regular homotopy in the class of generic immersions,
which does not change under opposite self-tangencies. Under direct self-tangencies of index $i$, $Z(f,\G)$ jumps 
by $-\big(f(i+1/2)-f(i-1/2)\big)$.
Under triple-point modifications of index $i$, 
it jumps by 
$-\frac12\big(f(i+3/2)-2f(i+1/2)+f(i-1/2)\big)$.
\end{rem}


\begin{thebibliography}{9}
\bibitem{A} V. I. Arnold, {\em Topological invariants of plane curves and caustics},
University Lecture Series 5, Providence, RI (1994).
\bibitem{H} H. Hopf, {\em \"{U}ber die Drehung der Tangenten und Sehenen ebener Kurven}, Compos. Math., 2, 50--62 (1935).
\bibitem{V} O. Viro, {\em Generic immersions of the circle to surfaces and the complex topology of real algebraic curves},  Topology of real algebraic varieties and related topics, Amer. Math. Soc. Transl., Ser. 2, Vol. 173, 231--252 (1996).
\end{thebibliography}
\end{document}